\documentclass[11pt]{amsart}

\usepackage{amsfonts, amstext, amsmath, amsthm, amscd, amssymb}
\usepackage{epsfig, graphics, psfrag, overpic}
\usepackage{color}
\usepackage[
pdfborderstyle={},
pdfborder={0 0 0}]{hyperref}

\renewcommand{\setminus}{{\smallsetminus}}

\newcommand{\RR}{{\mathbb{R}}}

\newcommand{\bdy}{{\partial}}

\newcommand{\abs}[1]{{\left\vert #1 \right\vert}}
\newcommand{\G}{{\mathbb{G}}}
\newcommand{\GA}{{\mathbb{G}_A}}

\newcommand{\GRA}{{\mathbb{G}'_A}}
\newcommand{\GRB}{{\mathbb{G}'_B}}
\newcommand{\Gs}{{\mathbb{G}_{\sigma}}}
\newcommand{\Grs}{{\mathbb{G}'_{\sigma}}}

\theoremstyle{plain}
\newtheorem{theorem}{Theorem}
\newtheorem{corollary}[theorem]{Corollary}
\newtheorem{lemma}[theorem]{Lemma}

\newtheorem*{namedtheorem}{\theoremname}
\newcommand{\theoremname}{testing}

\theoremstyle{definition}

\newtheorem{remark}[theorem]{Remark}

\begin{document}
\title{Fiber detection for state surfaces}

\author{David Futer}

\address{Department of Mathematics, Temple University,
Philadelphia, PA 19122}

\email{dfuter@temple.edu}

\thanks{{Supported in part by NSF grant DMS--1007221.}}

\thanks{ \today}

\begin{abstract}
Every Kauffman state $\sigma$ of a link diagram $D(K)$ naturally defines a state surface $S_\sigma$ whose boundary is $K$. For a \emph{homogeneous} state $\sigma$, we show that $K$ is a fibered link with fiber surface $S_\sigma$ if and only if an associated graph $\Grs$ is a tree. As a corollary, it follows that for an adequate knot or link, the second and next-to-last coefficients of the Jones polynomial are the obstructions to certain state surfaces being fibers for $K$.

This provides a dramatically simpler proof of a theorem of the author with Kalfagianni and Purcell.
\end{abstract}

\maketitle

\section{Introduction}\label{sec:intro}

In the 1930s, Seifert gave an algorithm that starts with a link diagram $D(K)$ and produces an orientable surface whose boundary is $K$ \cite{seifert:algorithm}. The algorithm works as follows. For every crossing of $D$, smooth the diagram near the crossing by following an orientation on $K$. This gives a disjoint union of circles in the projection plane. These circles bound a number of  disks, disjointly embedded in the ball below the projection plane. Then, these disks can be jointed by half-twisted bands at the crossings to give a surface $S \subset S^3$, such that $\bdy S = K$.

Seifert's construction has a natural generalization. At each crossing, there are two possible smoothings, or \emph{resolutions} of the crossing, as depicted in Figure \ref{fig:splicing}. A \emph{Kauffman state}  is a choice of $A$-- or $B$--resolution at each crossing. As in Seifert's construction,  a state $\sigma$ gives rise to a union of circles in the projection plane. These circles bound disjoint disks, which can be joined by half-twisted bands to give a \emph{state surface} $S_\sigma$. See Figure \ref{fig:state-surface}.

\begin{figure}[h]
\begin{center}
\begin{overpic}{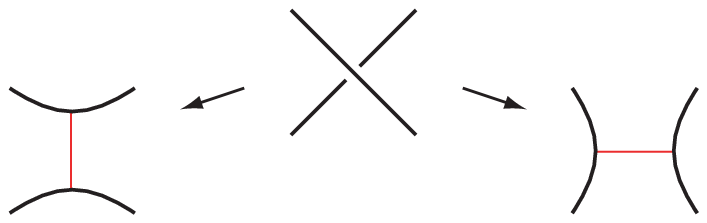}
\put(29,20){$A$}
\put(68,20){$B$}
\end{overpic}
\end{center}
\caption{$A$-- and
  $B$--resolutions at a crossing of $D$.\index{$A$--resolution}\index{$B$--resolution}}
\label{fig:splicing}
\end{figure}

Another  common example of a state surface is the two checkerboard surfaces of a diagram. The regions in the complement of $D(K)$ can be checkerboard colored, black and white. If the state $\sigma$ traces the boundaries of all the black regions, then joining these regions together produces the black checkerboard surface. Making the opposite choice for $\sigma$ produces the white surface. These checkerboard surfaces have been studied since the time of Tait in the 19th century; see Przytycki \cite{przytycki:survey} for a survey. In the special case of alternating diagrams, the checkerboard surfaces correspond to choosing the all--$A$ or all--$B$ state for $\sigma$.


The goal of this paper is to characterize when a state surface $S_\sigma$ is a fiber in a fibration of $S^3 \setminus K$ over $S^1$. For alternating diagrams, the strikingly simple answer is that the all--$A$ surface $S_A$ is a fiber if and only if $D$ is a connected sum of positive $2$--braids. (See Lemma \ref{lemma:basecase} below.) Stating our result in general requires a handful of definitions. 

If every crossing of $D(K)$ is resolved as in Figure \ref{fig:splicing}, according to a state $\sigma$, then the crossing point in the projection plane lies in a region complementary to the state circles of $\sigma$. These state circles partition the projection plane into regions. The state $\sigma$ is called \emph{homogeneous} if all crossings in the same region carry the same ($A$ or $B$) resolution. For example, the  state shown in Figure \ref{fig:state-surface} is homogeneous.  This notion was introduced by Cromwell \cite{cromwell:homogeneous} for the Seifert state, and easily extends to other states. 

\begin{figure}
\begin{center}
\begin{overpic}{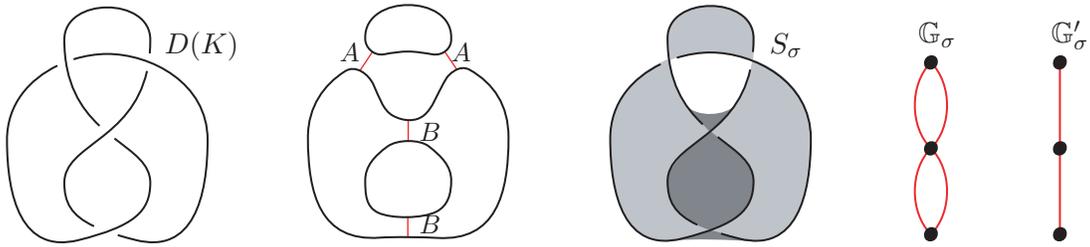}      
\put(15,18){$D(K)$}
\put(31.5,17.25){\small{$A$}}
\put(42,17.25){\small{$A$}}
\put(39,9.75){\small{$B$}}
\put(39,1){\small{$B$}}
\put(72,18){$S_\sigma$}
\put(86,19){$\Gs$}
\put(98.5,19){$\Grs$}
\end{overpic}
\end{center}
\caption{Left to right: a diagram $D(K)$. A homogeneous state $\sigma$, coming from Seifert's algorithm. The state surface $S_\sigma$ corresponding to $\sigma$. The graph $\Gs$ embeds into $S_\sigma$ as a spine. The reduced graph $\Grs$.}
\label{fig:state-surface}
\end{figure}

The choices that lead to a Kauffman state $\sigma$ can be conveniently encoded in a \emph{state graph} $\Gs$. This graph has one vertex for each state circle of $\sigma$. Each crossing $x$ of $D$ gives rise to an edge between the state circles at the resolution of $x$. (In Figures \ref{fig:splicing} and \ref{fig:state-surface}, these edges are shown in red, lighter than the link projection.) 
%
From the graph $\Gs$, we construct a reduced graph $\Grs$ by removing all duplicate edges between a pair of vertices. Our main result  is that  this reduced graph carries fibering information.

\begin{theorem}\label{thm:detection}
Let $\sigma$ be a homogeneous state of a link diagram $D(K)$. Then $S^3 \setminus K$ fibers over the circle with fiber $S_\sigma$ if and only if the reduced graph $\Grs$ is a tree.
\end{theorem}

The topology of state surfaces was recently studied by Ozawa \cite{ozawa}. For homogeneous states, he showed that the surface $S_\sigma$ is essential in $S^3 \setminus K$ if and only the state $\sigma$ is \emph{adequate}, meaning that $\Gs$ has no $1$--edge loops (equivalently, $\Grs$ has no $1$--edge loops). Since a tree has no loops of any length, all the states where $S_\sigma$ is a fiber must be adequate. 

In the special case where $\sigma$ is the all--$A$ or all--$B$ state, the graphs $\GRA$ and $\GRB$ are particularly worthy of attention due to their connection to the Jones polynomial of $K$. This is a Laurent polynomial invariant of $K$, which can be written in the form 
$$J_K(t)= \alpha t^k+ \beta t^{k-1}+ \ldots + \beta' t^{m+1}+
\alpha' t^m, $$
so that the second and next-to-last coefficients of $J_K(t)$ are
$\beta$ and $\beta'$, respectively. Stoimenow \cite{stoimenow:coeffs} and Dasbach and Lin \cite{dasbach-lin:head-tail} have observed that if $D$ is $A$--adequate (meaning the all--$A$ state is adequate), then $\abs{\beta'} = 1 - \chi(\GRA)$. Similarly, if $D$ is $B$--adequate, then $\abs{\beta} = 1 - \chi(\GRB)$. As a result, the following corollary immediately follows from Theorem \ref{thm:detection}.

\begin{corollary}\label{cor:beta-fiber}
Let $K$ be a link that admits a connected, $A$--adequate diagram $D$. Then  $S^3 \setminus K$ fibers over $S^1$ with fiber the state surface $S_A = S_A(D)$ if and only if $\beta'=0$.
\end{corollary}

In other words, the next-to-last Jones coefficient $\beta'$ is precisely the obstruction to $S^3 \setminus K$ being fibered with fiber surface $S_A$. Similarly, for $B$--adequate links, the second coefficient $\beta$ is the obstruction to $S^3 \setminus K$ being fibered with fiber surface $S_B$. 

Ever since Jones introduced his knot polynomial in the 1980s, it has been a tantalizing open problem to understand exactly what this invariant and its generalizations say about the topology of knot and link complements. The first rigorous relations of this sort have appeared in the last few years, in the work of Dasbach--Lin \cite{dasbach-lin:volumish}, Garoufalidis \cite{garoufalidis:jones-slopes}, and Futer--Kalfagianni--Purcell \cite{fkp:volume, fkp:jones-slopes, fkp:gutsjp}. Corollary \ref{cor:beta-fiber} is one of the more striking connections between quantum knot invariants and classical geometric topology to have been found so far.

Theorem \ref{thm:detection} and Corollary \ref{cor:beta-fiber} were first proved in the author's joint work with Kalfagianni and Purcell \cite[Theorem 5.21 and Corollary 9.16]{fkp:gutsjp}. However, the proof of \cite[Theorem 5.21]{fkp:gutsjp} appears over 80 pages into the monograph. That proof relies on the detailed study of a polyhedral decomposition of $S^3 \setminus S_\sigma$, and much effort is expended to show that the polyhedra have desirable properties \cite[Chapters 2--4]{fkp:gutsjp}. This polyhedral decomposition is also used to gain a detailed understanding of $I$--bundles in $S^3 \setminus S_\sigma$ and the hyperbolic geometry of $S^3 \setminus K$.

By contrast, the proof of Theorem  \ref{thm:detection} contained in this paper is short and direct. As a consequence, one also obtains a short and readily digestible proof of Corollary \ref{cor:beta-fiber}.

Our proof builds up the surface $S_\sigma$ inductively via Murasugi sums (see Figure \ref{fig:murasugi}), applying results of Gabai \cite{gabai:natural} to deduce fibering information. This inductive approach is very similar in spirit to the methods used by Ozawa \cite{ozawa}. It is also fruitfully exploited in a recent preprint of Gir\~{a}o to prove a fibering criterion for augmented links \cite{girao:fibering}.

\section{Proof}\label{sec:proof}

Before beginning the main proof of Theorem \ref{thm:detection}, we wish to dismiss a few special cases.  If $D(K)$ depicts an unknot with no crossings, then $\Grs$ is a single vertex, and the spanning disk $S_\sigma$ is a fiber of the solid torus $S^3 \setminus K$. Thus the theorem holds trivially.  If $D(K)$ is a split diagram, then $\Grs$ is a disconnected graph (which cannot be a tree), and $S^3 \setminus K$ is reducible (hence cannot be fibered). Again, the theorem holds trivially  in this case.

If the state $\sigma$ is not adequate, i.e.\ $\G_\sigma$ has a $1$--edge loop, the reduced graph $\Grs$ cannot be a tree. In this case, the surface $S_\sigma$ is non-orientable, hence cannot be a fiber. Thus the theorem holds for non-adequate states.

For the remainder of the paper, we work under the assumptions that $D(K)$ is connected and has at least one crossing, and that the state $\sigma$ is adequate.
With these simplifying assumptions, the proof of Theorem \ref{thm:detection} proceeds by induction on the number of cut vertices in the graph $\Grs$. Recall that a \emph{cut vertex} is a vertex that separates $\Grs$.

\smallskip

The base case of the induction involves prime, alternating diagrams. Recall that a link diagram $D(K)$ is called \emph{prime} if, for every simple closed curve $\gamma$ in the projection plane that intersects $D(K)$ transversely in two points, one of the two sides of $\gamma$ contains no crossings. In other words, $D(K)$ fails to be prime precisely when it is the connected sum of two non-trivial diagrams.

\begin{figure}

\begin{overpic}{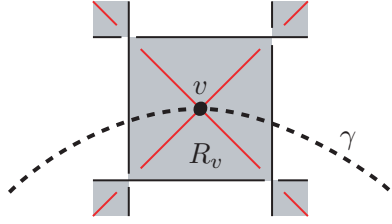}
\put(48,32){$v$}
\put(47,14){$R_v$}
\put(86,19){$\gamma$}
\end{overpic}

\caption{If $D(K)$ is an alternating diagram, and $v$ is a cut vertex of the checkerboard graph $\GA$, then the corresponding black region $R_v$ separates the diagram as a connected sum (with summands on opposite sides of $\gamma$).}
\label{fig:primeness}
\end{figure}

\begin{lemma}\label{lemma:primealt}
Let $\sigma$ be a homogeneous state of $D(K)$. Then $\Grs$ does not contain any cut vertices if and only if all of the following hold: $D(K)$ is prime and alternating, and $\sigma$ is the all--$A$ or all--$B$ state.
\end{lemma}

\begin{proof}
For the ``if'' direction, suppose that $D(K)$ is prime and alternating, and $\sigma$ is the all--$A$ or all--$B$ state. Without loss of generality, $\sigma$ is the all--$A$ state. Then $\Gs$ is the all--$A$ checkerboard graph of $D$.  The graph $\Gs = \GA$ naturally embeds as a spine of the (black) checkerboard surface $S_\sigma= S_A$, with one vertex in each black region of $D(K)$ and one edge running through each half-twisted band. 

Suppose, for a contradiction, that $v \in \GRA$ is a cut vertex. Then $v$ also separates the unreduced graph $\GA$. Since $\GA$ is embedded as a spine of the checkerboard surface $S_A$, the black region $R_v$ corresponding to $v$ must separate the checkerboard surface, hence also separate the diagram $D(K)$. In other words, there is a simple closed curve $\gamma$ in the projection plane such that the only black region met by $\gamma$ is $R_v$, and such that each component of $\RR^2 \setminus \gamma$ contains at least one crossing (these crossings correspond to edges of $\GA \setminus v$).  
This curve decomposes $D(K)$ as a connected sum, violating the hypothesis of primeness. See Figure \ref{fig:primeness}.

\begin{figure}[h]
\begin{overpic}{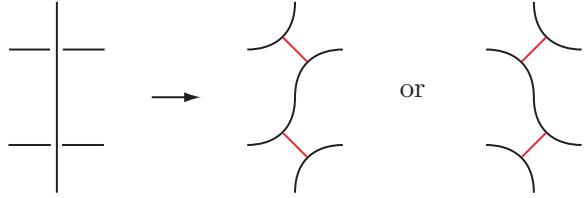}
\put(68,17){or}
\end{overpic}
\caption{If $D(K)$ is a non-alternating diagram, a homogeneous resolution of an over-over strand of $D$ must result in a state circle that has other circles on both sides. The same is true of an under-under strand.}
\label{fig:homogeneous}
\end{figure}

\smallskip

For the ``only if'' direction, suppose that $D(K)$ is not alternating. Since $\sigma$ is homogeneous,  some state circle of $\sigma$ (corresponding to a non-alternating segment of $D$) must have other state circles on both sides. See Figure \ref{fig:homogeneous}. The corresponding vertex of $\Gs$ will separate $\Gs$, and also $\Grs$. Thus, if $\Grs$ contains no cut vertices, $D(K)$ must be alternating.

If $D$ is alternating but $\sigma$ is not the all--$A$ or all--$B$ state, then some state circle fails to follow the boundary of a region of $D$. Just as in Figure \ref{fig:homogeneous}, this implies that this state circle has other circles on both sides, hence $\Grs$ has a cut vertex.

Finally, if $D$ is not prime, then there is a simple closed curve $\gamma$ that meets only one white region and only one black region, with crossings on both sides of $\gamma$. One of these regions corresponds to a vertex of $\Gs$ (depending on whether $\sigma$ is the all--$A$ or all--$B$ state), and this vertex separates $\Grs$. Therefore, if $\Grs$ has no cut vertices, $D(K)$ must be  prime and alternating, with $\sigma$ the all--$A$ or all--$B$ state.
\end{proof}

The base case of the induction is the following lemma.

\begin{lemma}\label{lemma:basecase}
Suppose $D(K)$ is a prime, alternating diagram with at least one crossing. Then the following are equivalent for the all--$A$ state: 
\begin{enumerate}
\item\label{i:gra-tree} $\GRA$ is a tree. 
\item\label{i:gra-stick} $\GA$ has exactly two vertices.
\item\label{i:braid} $D(K)$ is a positive $2$--braid.
\item\label{i:sa-fiber} The checkerboard surface $S_A$ is a fiber in $S^3 \setminus K$.
\end{enumerate}
The same equivalence holds for $S_B$, $\GRB$, and negative $2$--braids.
\end{lemma}

\begin{proof}
Suppose $\GRA$ is a tree. If it has only one vertex, then there are no edges, contradicting the hypothesis that $D(K)$ has crossings. If it has three or more vertices, some vertex will be separating, contradicting Lemma \ref{lemma:primealt}. Thus $\GA$ has exactly two vertices, giving $\eqref{i:gra-tree} \Rightarrow \eqref{i:gra-stick}$.

If $\GA$ has two vertices, then $D(K)$ has exactly two black regions, connected to each other at each crossing. This is the diagram of a positive $2$--braid. Conversely, if $D(K)$ is a positive $2$--braid, then $\GRA$ is a stick with two vertices, which is  a tree. Thus  $\eqref{i:gra-tree} \Leftrightarrow \eqref{i:gra-stick} \Leftrightarrow \eqref{i:braid}$.

For $\eqref{i:braid} \Rightarrow \eqref{i:sa-fiber}$, assume that $D(K)$ is a positive $2$--braid. Then $K$ is a $(2,n)$ torus knot, and it is well-known that $S_A$ is a fiber \cite{stallings:fibered}. 

%

For $\eqref{i:sa-fiber} \Rightarrow \eqref{i:braid}$, recall that if $D$ is alternating, prime, and not a $2$--braid, then $S^3 \setminus K$ admits a hyperbolic structure \cite{menasco:alternating}. Then, Adams shows that the checkerboard surface $S_A$ is not a fiber \cite[Theorem 1.9]{adams:quasifuchsian}. 
\end{proof}

\begin{remark}
There is also a more direct proof that $ \eqref{i:braid}  \Leftrightarrow \eqref{i:sa-fiber}$, which does not rely on hyperbolic geometry.  Instead, this alternate argument relies on Menasco's decomposition of $S^3 \setminus K$ into two ideal polyhedra \cite{menasco:polyhedra}. The $1$--skeleton of each polyhedron is isomorphic (as a planar graph) to the projection graph of $D(K)$. In particular, the faces of the polyhedra can be $2$--colored: the union of all the black faces is the checkerboard surface $S_A$, while the union of all the white faces is the checkerboard surface $S_B$. In particular, the manifold $S^3 \setminus S_A$ can be obtained by gluing the two polyhedra along white faces only.

For $\eqref{i:braid} \Rightarrow \eqref{i:sa-fiber}$, suppose $D$ is a positive $2$--braid. Then every white face is a bigon. In other words, every polyhedron is combinatorially a prism $P \times I$, where $P$ is an ideal polygon, and the lateral faces are the white bigons. The product structure of $P \times I$ extends as we glue the two polyhedra along their lateral faces, implying that $S^3 \setminus S_A \cong S_A \times I$, hence $S_A$ is a fiber.

For $\eqref{i:sa-fiber} \Rightarrow \eqref{i:braid}$, suppose  $S^3 \setminus S_A \cong S_A \times I$. Then, \cite[Lemma 4.17]{fkp:gutsjp} shows that Menasco's polyhedral decomposition must ``see'' this product structure: every white face must be a product of an ideal edge with $I$, i.e., an ideal bigon. (Note that when the proof of that lemma is applied to Menasco's well-known polyhedral decomposition, it becomes self-contained, and does not require the machinery developed in \cite{fkp:gutsjp}.) If every white $B$--region of $D(K)$ is a bigon, these bigons must be joined end to end, implying that $D(K)$ is a positive $2$--braid.
\end{remark}

We are now ready to complete the proof of the main theorem. 

\begin{figure}[b]

\begin{center}
\begin{overpic}{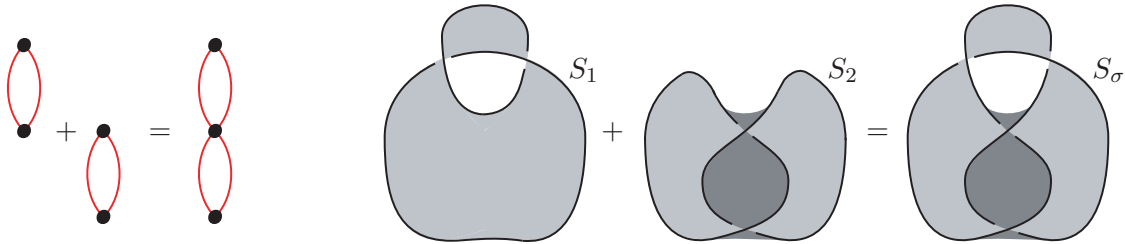}
\put(4.5,9.5){$+$}
\put(13,9.5){$=$}
\put(54,9.5){$+$}
\put(78,9.5){$=$}
\put(51,15){$S_1$}
\put(74.5,15){$S_2$}
\put(98.5,15){$S_\sigma$}
\end{overpic}
\end{center}

\caption{Left: The graph $\Gs$ decomposes as a union of subgraphs $\G_1$ and $\G_2$, joined along the cut vertex $v$. Right: The corresponding decomposition of $S_\sigma$ as a Murasugi sum of state surfaces $S_1$ and $S_2$.}
\label{fig:murasugi}
\end{figure}

\begin{proof}[Proof of Theorem \ref{thm:detection}]
We proceed by induction on $n$, where $n$ is the number of cut vertices in $\Grs$. For the base case, let $n=0$, and recall the running assumption that $D(K)$ has at least one crossing. Then Lemma \ref{lemma:primealt} says that the diagram $D(K)$ is prime and alternating, and $\sigma$ is the all--$A$ or all--$B$ state. By Lemma \ref{lemma:basecase}, $S_\sigma$ is a fiber if and only if $\Grs$ is a tree, as desired.

For the inductive step, suppose $n > 0$ and $v$ is a cut vertex of $\Grs$. Then $\Grs = \G'_1 \cup_v \G'_2$, where $\G'_1$ and $\G'_2$ are subgraphs that are disjoint except at $v$. The unreduced graph $\Gs$, which has the same adjacency relations as $\Grs$, also decomposes as a union of subgraphs $\G_1$ and $\G_2$ that are disjoint except at $v$. See Figure \ref{fig:murasugi}, left.
On the diagrammatic side, $D(K)$ decomposes as the Murasugi sum of diagrams $D(K_1)$ and $D(K_2)$, and the state surface $S_\sigma$ decomposes as the Murasugi sum of state surfaces $S_1$ and $S_2$. See Figure \ref{fig:murasugi}.

Note that the Kauffman state $\sigma_i$ that gives rise to $S_i$ is obtained by restricting $\sigma$ to the crossings of $D(K_i)$. Thus each $S_i$ is the state surface of a homogeneous state, with reduced graph $\G'_i$. Note as well that each $\G'_i$ has fewer than $n$ cut vertices. Thus, by the inductive hypothesis, $S_i$ is a fiber in $S^3 \setminus K_i$ if and only if $\G'_i$ is a tree.

Now, we recall Gabai's theorem \cite{gabai:natural}: $S_\sigma$ is a fiber if and only if each $S_i$ is a fiber. Clearly, $\Grs$ is a tree if and only if each $\G'_i$ is a tree. This completes the proof.
\end{proof}

\bibliographystyle{hamsplain}
\bibliography{biblio.bib}

\end{document}